\documentclass[]{amsart}   
\usepackage{amsmath}
\usepackage[mathscr]{eucal} 
\usepackage{amssymb}
\usepackage{latexsym}
\usepackage{amsthm} 
\theoremstyle{plain}
\newtheorem{theorem}{Theorem}[section]
\newtheorem{lemma}[theorem]{Lemma}

\newtheorem{proposition}[theorem]{Proposition}

\newtheorem{corollary}[theorem]{Corollary}

\newcommand{\supp}{\mathop{\mathrm{supp}}\nolimits}

\theoremstyle{definition} 

\newtheorem{definition}[theorem]{Definition}
\numberwithin{equation}{section}

\theoremstyle{remark}

\def\Xint#1{\mathchoice
{\XXint\displaystyle\textstyle{#1}}%
{\XXint\textstyle\scriptstyle{#1}}%
{\XXint\scriptstyle\scriptscriptstyle{#1}}%
{\XXint\scriptscriptstyle\scriptscriptstyle{#1}}%
\!\int}
\def\XXint#1#2#3{{\setbox0=\hbox{$#1{#2#3}{\int}$}
\vcenter{\hbox{$#2#3$}}\kern-.5\wd0}}

\def\dashint{\Xint-}

\title[Vector valued inequalities]
{Vector valued inequalities and Littlewood-Paley operators on 
Hardy spaces}  
\author{Shuichi Sato} 
 
\begin{document} 

\address{Department of Mathematics,
Faculty of Education, Kanazawa University, Kanazawa 920-1192, Japan}
\email{shuichi@kenroku.kanazawa-u.ac.jp}
\begin{abstract} 
We prove certain vector valued inequalities on $\Bbb R^n$ related to 
Littlewood-Paley theory. They can be used in proving 
characterization of the Hardy spaces in terms of  Littlewood-Paley operators by methods of real analysis. 
\end{abstract}
  \thanks{
  \endgraf
    }
\thanks{The author is partly supported
by Grant-in-Aid for Scientific Research (C) No. 25400130, Japan Society for the  Promotion of Science.}

\maketitle 

\section{Introduction}  

We consider the Littlewood-Paley function on $\Bbb R^n$ defined by 
\begin{equation}\label{lpop}
g_{\varphi}(f)(x) = \left( \int_0^{\infty}|f*\varphi_t(x)|^2
\,\frac{dt}{t} \right)^{1/2},    
\end{equation} 
where $\varphi_t(x)=t^{-n}\varphi(t^{-1}x)$.   
We assume that $\varphi \in L^1(\Bbb R^n)$  and   
\begin{equation}\label{cancell}
\int_{\Bbb R^n} \varphi (x)\,dx = 0.  
\end{equation} 
\par 
If we further assume that $|\varphi(x)|\leq C(1+|x|)^{-n-\epsilon}$ for 
some $\epsilon>0$, then we have 
$$\|g_{\varphi}(f)\|_p\leq C_p\|f\|_p, \quad 1<p<\infty,  $$  
where $\|f\|_p=\|f\|_{L^p}$ (see \cite{Sa2} and also \cite{BCP} for 
an earlier result).  The reverse inequality also holds if a certain
 non-degeneracy condition on $\varphi$ is assumed in addition (see 
 \cite[Theorem 3.8]{H} and also \cite{Sa}).  
 This is the case for $g_Q$ with $Q(x)= [(\partial/\partial t) P(x,t)]_{t=1}$, 
where $P(x,t)$ is the Poisson kernel associated with the upper half space 
$\Bbb R^n \times (0, \infty)$ defined by 
$$P(x,t)=c_n \frac{t}{(|x|^2+t^2)^{(n+1)/2}} $$ 
with $c_n=\pi^{-(n+1)/2}\Gamma((n+1)/2)$ (see \cite[Chap. I]{SW}). 
Here we recall that $\hat{Q}(\xi)=-2\pi|\xi|e^{-2\pi|\xi|}$, 
where the Fourier transform is defined as 
$$
\hat{f}(\xi)=\mathscr F(f)(\xi)=\int_{\Bbb R^n} f(x)e^{-2\pi i\langle x,\xi
\rangle}\, dx, \quad \langle x,\xi\rangle=x_1\xi_1+\dots +x_n\xi_n. 
$$  
\par 
Furthermore,  it is known that 
\begin{equation}\label{hpequiv}
c_1\|f\|_{H^p}\leq \|g_Q(f)\|_{p}\leq c_2\|f\|_{H^p} 
\end{equation}   
for $f\in H^p(\Bbb R^n)$ (the Hardy space), $0<p<\infty$, 
where   $c_1, c_2$ are positive constants 
(see \cite{FeS2} and also \cite{U}).  
Recall that a tempered distribution $f$ belongs to 
$H^p(\Bbb R^n)$ if 
$\|f\|_{H^p}=\|f^*\|_p<\infty$, where $f^*(x)=\sup_{t>0}|\Phi_t*f(x)|$.  
Here  $\Phi$ is in $\mathscr S(\Bbb R^n)$ and satisfies $\int \Phi(x)\, dx=1$, 
where $\mathscr S(\Bbb R^n)$ denotes the Schwartz 
class of rapidly decreasing smooth functions on $\Bbb R^n$; 
it is known that any other choice of such $\Phi$ gives an equivalent norm 
(see \cite{FeS2}). 
\par 
In this note we are concerned with the first inequality of \eqref{hpequiv} 
for $0<p\leq 1$.  
A proof of the inequality was given by Uchiyama \cite{U}. The proof is 
based on  real analysis methods and does not use special properties of 
the Poisson kernel such as  harmonicity, a semigroup property.  
Consequently,  \cite{U} can also prove 
\begin{equation}\label{reverse}
\|f\|_{H^p}\leq c\|g_\varphi(f)\|_p,   \quad 0<p\leq 1, 
\end{equation}
for $\varphi \in \mathscr S(\Bbb R^n)$ satisfying \eqref{cancell} and  
a suitable non-degeneracy condition.  
Also, a relation between Hardy spaces on homogeneous groups and 
Littlewood-Paley functions associated with the heat kernel  
 can be found in \cite[Chap. 7]{FoS}. 
\par 
On the other hand, 
it is known and would be seen by applying an easier version of our arguments 
in the following that  
the Peetre maximal function $F^{**}_{N,R}$ can be used along with familiar   
methods to prove \eqref{reverse} when $\varphi \in \mathscr S(\Bbb R^n)$ 
with a  non-degeneracy condition and with the condition $\supp(\hat{\varphi}) 
\subset \{a_1\leq |\xi|\leq a_2\}$, $a_1, a_2>0$, 
where  for a function $F$ on $\Bbb R^n$ and positive real numbers $N, R$,  
the maximal function is defined as 
\begin{equation}\label{pmax}
F^{**}_{N,R}(x)=\sup_{y\in \Bbb R^n}\frac{|F(x-y)|}{(1+R|y|)^N}   
\end{equation} 
(see \cite{P}).  
\par 
The purpose of this note is to prove \eqref{reverse} for a class of functions 
$\varphi$ including $Q$ and a general $\varphi \in  \mathscr S(\Bbb R^n)$, 
without the restriction on $\supp(\hat{\varphi})$ above, 
with \eqref{cancell} and an admissible non-degeneracy condition 
(Corollary \ref{C3.1}) as an 
application of a vector valued inequality which will be shown by using 
the maximal function $F^{**}_{N,R}$ (see Proposition \ref{T2.3},
 Theorem \ref{C2.10} below).  
The proof of Proposition \ref{T2.3} consists partly in 
further developing methods of \cite[Chap. V]{ST} and it admits some weighted 
inequalities.   
Theorem \ref{C2.10} follows from Proposition \ref{T2.3}.  
Our proofs of  Proposition \ref{T2.3} and 
Corollary \ref{C3.1} are fairly 
straightforward and they will be expected to extend to some other situations.  
\par  
In Section 2, Proposition \ref{T2.3} will be formulated in a general form, 
while 
Theorem \ref{C2.10} will be stated in a more convenient form for the 
application to the proof of Corollary \ref{C3.1}.   
In Section 3, we shall apply Theorem \ref{C2.10} and an atomic decomposition 
for Hardy spaces to prove Corollary \ref{C3.1}.  
Finally, in Section 4, we shall give 
proofs of Lemmas \ref{L2.1} and \ref{L2.5} in Section 2 from \cite{ST} and 
\cite{P}, respectively, for completeness; the lemmas will be needed in proving 
Proposition \ref{T2.3}.

\section{Vector valued inequalities}  
Let $\varphi^{(j)}$, $j=1, 2, \dots, M$, be functions in $L^1(\Bbb R^n)$  
satisfying the non-degeneracy condition 
\begin{equation}\label{nondegeneracy} 
\inf_{\xi\in \Bbb R^n\setminus\{0\}}\sup_{t>0}\sum_{j=1}^M
|\mathscr F(\varphi^{(j)})(t\xi)|> c  
\end{equation}
for some positive constant $c$. 
We write $\varphi=(\varphi^{(1)}, \dots, \varphi^{(M)})$,  $\hat{\varphi}
=(\mathscr F(\varphi^{(1)}), \dots, \mathscr F(\varphi^{(M)}))$.  

\begin{lemma}\label{L2.1} 
Let $\varphi^{(j)}$, $j=1, 2, \dots, M$, be functions in 
$L^1(\Bbb R^n)$ satisfying  \eqref{nondegeneracy}.  Then, there 
exist $b_0\in (0,1)$ and positive numbers $r_1, r_2$ with $r_1<r_2$ such that 
if $b\in [b_0, 1)$,  we can find $\eta=(\eta^{(1)}, \dots, \eta^{(M)})$ which 
satisfies the following$:$  
\begin{enumerate}
\item[(1)] $\eta\in C^\infty(\Bbb R^n)$, where $\eta\in C^k(U)$ 
means  $\eta^{(j)}\in C^k(U)$   for all $1\leq j\leq M; $  
\item[(2)] $\supp \mathscr F(\eta^{(j)}) \subset \{r_1<|\xi|<r_2\}, 
1\leq j\leq M;$  
\item[(3)] each $\mathscr F(\eta^{(j)})$ is continuous, $1\leq j\leq M; $ 
\item[(4)]
$\sum_{j=-\infty}^\infty \langle \hat{\varphi}(b^j\xi), 
\hat{\eta}(b^j\xi)\rangle =1$ \quad for $\xi\in \Bbb R^n\setminus\{0\}$, 
where $\langle z,w\rangle=\sum_{j=1}^M z_jw_j$, $z, w\in \Bbb C^M$ $($the 
Cartesian product of $M$ copies of the set of complex numbers$)$. 
\end{enumerate}  
Further, if $\hat{\varphi}\in C^k(\Bbb R^n\setminus\{0\})$, then $\hat{\eta}
\in C^k(\Bbb R^n)$.   
\end{lemma}  
See \cite[Chap. V]{ST} and also \cite{CT}. 
\par 
We assume that $M=1$ for simplicity. 
Suppose that $\psi \in L^1(\Bbb R^n)$  and there exist  
$\Theta\in C^\infty(\Bbb R^n)$ and $A\geq 1$ such that 
\begin{equation}\label{nearorigin} 
\hat{\psi}(\xi)=\hat{\varphi}(\xi)\Theta(\xi)   
\quad \text{on $\{|\xi|<r_2A^{-1} \}$. } 
\end{equation}
Suppose that $b\in [b_0,1)$ and let $\eta$ be as in Lemma \ref{L2.1} 
with $M=1$.  
For $J>0$, define $\zeta_J$ by 
\begin{equation}\label{2.3+}
\hat{\zeta}_J(\xi)=1-\sum_{j: b^j\leq J} \hat{\varphi}(b^j\xi)
\hat{\eta}(b^j\xi).    
\end{equation} 
We note that $\supp(\hat{\zeta}_J)\subset \{|\xi|\leq r_2J^{-1}\}$, 
$\hat{\zeta}_J=1$ in $\{|\xi|<r_1J^{-1}\}$.  By \eqref{nearorigin} it follows 
that 
\begin{align*} 
\hat{\psi}(\xi)&=\sum_{j: b^j\leq A}
\hat{\psi}(\xi)\hat{\varphi}(b^j\xi) \hat{\eta}(b^j\xi)
+ \hat{\zeta}_A(\xi)\hat{\psi}(\xi)  
\\ 
&=\sum_{j: b^j\leq A} \hat{\varphi}(b^j\xi)\mathscr F(\alpha^{(b^j)})(b^j\xi) 
+  \hat{\varphi}(\xi) \hat{\beta}(\xi), 
\end{align*}
where 
$\alpha^{(b^j)}(x)=\psi_{b^{-j}}*\eta(x)$ 
and 
$\hat{\beta}(\xi)=\hat{\zeta}_A(\xi)\Theta(\xi)$.   
\par 
 Let $E(\psi,f)(x,t)=f*\psi_t(x)$, $f\in \mathscr S(\Bbb R^n)$. 
Then we have 
\begin{equation}\label{ineq1} 
|E(\psi,f)(x,t)|\leq   \sum_{j: b^j\leq A}|E(\alpha^{(b^j)}*\varphi,f)(x,b^jt)|
+|E(\beta*\varphi,f)(x,t)|.  
\end{equation} 
Also, let $ E_\psi(x,t)=E(\psi,f)(x,t)$,  when $f$ is fixed.    
\par 
Define 
\begin{equation}\label{czero} 
C_0(\psi,t,L,x)= (1+|x|)^L\left|\int \hat{\psi}(t^{-1}\xi)\hat{\eta}(\xi)
e^{2\pi i\langle x, \xi\rangle}\, d\xi\right|, \quad t>0, L\geq 0. 
\end{equation} 
Consequently, 
\begin{equation*} 
|\alpha^{(b^j)}_s(x)|= C_0(\psi,b^j,L,x/s)s^{-n}(1+|x|/s)^{-L}     
\end{equation*}  
for $j\in \Bbb Z$ (the set of integers). 
Likewise, we have  
\begin{equation*} 
|\beta_s(x)|=D(\Theta,A, L,x/s)s^{-n}(1+|x|/s)^{-L},     
\end{equation*}  
where 
\begin{equation}\label{d}
D(\Theta,J, L,x)= (1+|x|)^L\left|\int \hat{\zeta}_J(\xi)\Theta(\xi)e^{2\pi i
\langle x, \xi\rangle}\, d\xi\right|.   
\end{equation}  
Here $\hat{\zeta}_J$ is as in \eqref{2.3+}. 
We also write $C(\psi,j, L,x)=C_0(\psi,b^j,L,x)$, $j\in \Bbb Z$.    
Let  
\begin{gather}\label{c} 
C(\psi,j, L)= \int_{\Bbb R^n} C(\psi,j, L,x)\, dx, \quad j\in \Bbb Z,    
\\ 
\label{d2} 
D(\Theta,J, L)= \int_{\Bbb R^n} D(\Theta,J, L,x)\, dx.  
\end{gather} 
We also write 
$C(\psi,j, L)=C_\varphi(\psi,j, L)$, 
$D(\Theta,J, L)=D_\varphi(\Theta,J, L)$  
to indicate that these quantities are based on $\varphi$.  
See Lemma \ref{L2.8} below for a sufficient condition which implies  
$C(\psi,j, L)<\infty$, $D(\Theta,J, L)<\infty$.   
\par 
The maximal function in \eqref{pmax} is used in the following result.  
\begin{lemma}\label{L2.2} 
Let $\varphi, \psi\in L^1(\Bbb R^n)$. Suppose that 
 $\varphi$ satisfies \eqref{nondegeneracy}. Let $b\in [b_0,1)$.  
We assume that $\psi$ and $\varphi$ are related by \eqref{nearorigin} 
with $\Theta\in C^\infty(\Bbb R^n)$ and $A\geq 1$. 
Let $N>0$.   Then for $f\in \mathscr S(\Bbb R^n)$, we have 
\begin{multline}\label{ineq2}
|E(\psi,f)(x,t)| 
\leq C\sum_{j: b^j\leq A}
C(\psi,j,N) E(\varphi,f)(\cdot,b^jt)^{**}_{N, (b^jt)^{-1}}(x) 
\\ 
+CD(\Theta,A,N) E(\varphi,f)(\cdot,t)^{**}_{N, t^{-1}}(x);  
\end{multline} 
\begin{multline}  
\label{ineq3} 
E(\psi,f)(\cdot,t)^{**}_{N, t^{-1}}(x) 
\leq C\sum_{j: b^j\leq A}
C(\psi, j,N)b^{-jN} E(\varphi,f)(\cdot,b^jt)^{**}_{N, (b^jt)^{-1}}(x) 
\\ 
+CD(\Theta,A, N) E(\varphi,f)(\cdot,t)^{**}_{N, t^{-1}}(x).     
\end{multline}  
\end{lemma} 
\begin{proof} 
Using \eqref{ineq1},  we see that  
\begin{multline*}  
|E_\psi(z,t)| 
\\ 
\leq C \sum_{j: b^j\leq A}\int |E_\varphi(y,b^jt)|
\left(1+\frac{|z-y|}{b^jt}\right)^{-N}C(\psi, j, N,(z-y)/(b^jt)) 
(b^jt)^{-n}\, dy 
\\ 
+C\int |E_\varphi(y,t)|
\left(1+\frac{|z-y|}{t}\right)^{-N}D(\Theta,A, N,(z-y)/t) 
t^{-n}\, dy.    
\end{multline*} 
If we multiply both sides of the inequality by $(1+|x-z|/t)^{-N}$ and 
observe that 
\begin{equation*} 
\left(1+\frac{|z-y|}{b^jt}\right)^{-N}
\left(1+\frac{|x-z|}{t}\right)^{-N}\leq C_{A,N}b^{-Nj}
\left(1+\frac{|x-y|}{b^jt}\right)^{-N}
\end{equation*}  
for all $x, y, z\in \Bbb R^n$ and $t>0$ under the condition $b^j\leq A$,  
then  we see that  
\begin{align*}  
&|E_\psi(z,t)|(1+|x-z|/t)^{-N}
\\ 
&\leq C \sum_{j: b^j\leq A}b^{-Nj}\int |E_\varphi(y,b^jt)|
\left(1+\frac{|x-y|}{b^jt}\right)^{-N}C(\psi, j, N,(z-y)/(b^jt))
(b^jt)^{-n}\, dy 
\\ 
&\phantom{ \leq\, }+ C\int |E_\varphi(y,t)|
\left(1+\frac{|x-y|}{t}\right)^{-N}D(\Theta,A,N,(z-y)/t)
t^{-n}\, dy, 
\end{align*}
and hence 
\begin{align*}  
&|E_\psi(z,t)|(1+|x-z|/t)^{-N}
\\ 
&\leq C \sum_{j: b^j\leq A}b^{-Nj}E_\varphi(\cdot,b^jt)^{**}
_{N, (b^jt)^{-1}}(x)\int C(\psi, j, N,(z-y)/(b^jt))
(b^jt)^{-n}\, dy 
\\ 
&\phantom{ \leq\, }+C E_\varphi(\cdot,t)^{**}
_{N, t^{-1}}(x)\int D(\Theta,A, N,(z-y)/t)t^{-n}\, dy 
\\ 
&\leq C \sum_{j: b^j\leq A}C(\psi,j,N)b^{-Nj}E_\varphi(\cdot,b^jt)^{**}
_{N, (b^jt)^{-1}}(x) 
+CD(\Theta,A,N) E_\varphi(\cdot,t)^{**}_{N, t^{-1}}(x). 
\end{align*}
The estimate \eqref{ineq3} follows by taking the supremum in $z$ over 
$\Bbb R^n$. 
The proof of \eqref{ineq2} is easier; putting $z=x$ and arguing as above, 
we get \eqref{ineq2}.   
\end{proof} 
\par 
Let $ \varphi \in L^1(\Bbb R^n)$. Suppose that $\varphi$ satisfies 
\eqref{nondegeneracy}. 
Let $L>0$. 
We consider the following conditions.  
\begin{gather} \label{alpha} 
\varphi\in C^1(\Bbb R^n), \quad \partial_{k}\varphi\in L^1(\Bbb R^n), 
\quad 1\leq k\leq n;  
\\  \label{beta} 
|\hat{\varphi}(\xi)|\leq C|\xi|^\epsilon \quad \text{for some $\epsilon>0;$} 
\\ \label{ineq5} 
\sup_{j\geq 0}C_\varphi(\nabla\varphi,j,L)b^{-jL-\epsilon j}<\infty 
\quad 
\text{for some $\epsilon>0$, together with \eqref{alpha};   } 
\\ \label{ineq5+} 
D_\varphi(L)<\infty, 
\quad 
\text{ with \eqref{alpha},  }
\end{gather}  
where we write  
$\nabla\varphi=(\partial_{1}\varphi, \dots , \partial_{n}\varphi)$, 
$\partial_k=\partial_{x_k}=\partial/\partial_{x_k}$  and 
$C_\varphi(\nabla\varphi,j,L)=\sum_{k=1}^n 
C_\varphi(\partial_{k}\varphi, j,L)$; 
also we define 
$D_\varphi(L)=\sum_{k=1}^n 
D_\varphi(\Xi_k,1,L)$ 
by taking $\Theta(\xi)=\Xi_k(\xi)=2\pi i\xi_k$ and $J=1$ in \eqref{d2}. 
We note that \eqref{alpha} 
implies the following (with $\epsilon=1$):     
\begin{equation}\label{beta+}
|\hat{\varphi}(\xi)|\leq 
C|\xi|^{-\epsilon}\quad \text{for some $\epsilon>0.$} 
\end{equation} 
Let $\psi \in L^1(\Bbb R^n)$. We assume that $\psi$ is  related to $\varphi$ 
as in \eqref{nearorigin} with $\Theta \in C^\infty(\Bbb R^n)$ and $A\geq 1$. 
We also consider the conditions:  
\begin{gather}\label{ineq4}
\sup_{j: b^j\leq A}C_\varphi(\psi,j,L)b^{-\epsilon j}<\infty \quad 
\text{for some $\epsilon>0$;} 
\\   
\label{ineq4+} 
D_\varphi(\Theta,A,L)<\infty.   
\end{gather}
\par 
Let $M$ be the Hardy-Littlewood maximal operator 
$$M(f)(x)=\sup_{x\in B}|B|^{-1}\int_B|f(y)|\,dy,   $$ 
where the supremum is taken over all balls $B$ in $\Bbb R^n$ such that 
$x\in B$ and $|B|$ denotes the Lebesgue measure of $B$.    
Let $1<p< \infty$. We recall that 
a weight function $w$ belongs to the weight class $A_p$ of Muckenhoupt on 
$\Bbb R^n$ if 
 $$[w]_{A_p}=  
 \sup_B \left(|B|^{-1} \int_B w(x)\,dx\right)\left(|B|^{-1} \int_B
w(x)^{-1/(p-1)}dx\right)^{p-1} < \infty, $$
where the supremum is taken over all balls $B$ in $\Bbb R^n$ 
Also, we recall that a weight function $w$ is in the class $A_1$ if 
 $M(w)\leq Cw$ almost everywhere. 
The infimum of all such $C$ is denoted by $[w]_{A_1}$.  
\par 
For a weight $w$,  the weighted $L^p$ norm is defined as 
$$\|f\|_{p,w}=\left(\int_{\Bbb R^n} |f(x)|^p w(x)\, dx\right)^{1/p}. $$
We have the following vector value inequality. 
\begin{proposition}\label{T2.3} 
Let $\varphi \in L^1(\Bbb R^n)$. 
We assume that 
$\varphi$ satisfies \eqref{nondegeneracy} with $M=1$.  
 Let $N>0$,   $n/N<p, q<\infty$ and $w \in A_{pN/n}$. 
Suppose that $\varphi$ satisfies \eqref{alpha}, \eqref{beta} and 
\eqref{ineq5}, \eqref{ineq5+} with $L=N$.  
Let $\psi\in L^1(\Bbb R^n)$. 
Suppose that $\psi$ is related to $\varphi$ as in \eqref{nearorigin} with 
 $\Theta\in C^\infty(\Bbb R^n)$, $A\geq 1$ and   
\eqref{ineq4}, \eqref{ineq4+} hold with $L=N$. 
 Then 
$$\left\|\left(\int_0^\infty|f*\psi_t|^q\, \frac{dt}{t}
\right)^{1/q}\right\|_{p,w} 
\leq C\left\|\left(\int_0^\infty|f*\varphi_t|^q\, 
\frac{dt}{t}\right)^{1/q}\right\|_{p,w}   $$  
for $f\in \mathscr S(\Bbb R^n)$ with a positive constant $C$ independent of 
$f$.  
\end{proposition}  
\par 
We need the next result to show  Proposition \ref{T2.3}.  
\begin{lemma}\label{L2.4}  
Suppose that $0<q<\infty$, $N>0$ and 
 that $\varphi\in L^1(\Bbb R^n)$ satisfies \eqref{nondegeneracy}, 
\eqref{alpha}, \eqref{beta} and \eqref{ineq5}, \eqref{ineq5+} with $L=N$. Then 
\begin{equation*}\label{ineq10}
\int_0^\infty E(\varphi,f)(\cdot,t)^{**}_{N, t^{-1}}(x)^q \, \frac{dt}{t} \leq 
C\int_0^\infty M(|f*\varphi_t|^r)(x)^{q/r}\, \frac{dt}{t}, \quad r=n/N.   
\end{equation*}  
\end{lemma} 
 We need the following in proving Lemma \ref{L2.4}. 
\begin{lemma}[see \cite{P}]\label{L2.5}  
If $F\in C^1(\Bbb R^n)$ and $R>0$,  $r>0$,  then  
$$F^{**}_{N,R}(x)\leq C\delta^{-N} M(|F|^r)(x)^{1/r} +C\delta R^{-1}
|\nabla F|^{**}_{N,R}(x)  $$  
for all $\delta \in (0,1]$, where $N=n/r$ and the constant $C$ is 
 independent of $\delta$ and $R$. 
\end{lemma}
\begin{proof}[Proof of Lemma $\ref{L2.4}$.]   
By Lemma \ref{L2.5} we have   
\begin{equation}\label{ineq7} 
E(\varphi,f)(\cdot,t)^{**}_{N, t^{-1}}(x)\leq C\delta^{-N} 
M(|f*\varphi_t|^r)(x)^{1/r} +
C\delta |f*(\nabla\varphi)_t|^{**}_{N,t^{-1}}(x),  
\end{equation}  
where $f*(\nabla\varphi)_t=(f*(\partial_{1}\varphi)_t, \dots , 
f*(\partial_{n}\varphi)_t)$, $r=n/N$. 
We apply \eqref{ineq3} of Lemma \ref{L2.2} with $\psi=\partial_{k}\varphi$,  
$\Theta(\xi)=2\pi i\xi_k$, $A=1$ in \eqref{nearorigin}. Then 
\begin{multline*}  
|f*(\nabla\varphi)_t|^{**}_{N,t^{-1}}(x)
\\ 
\leq C\sum_{j\geq 0}
C_\varphi(\nabla\varphi, j,N)b^{-jN} E(\varphi,f)(\cdot,b^jt)^{**}
_{N, (b^jt)^{-1}}(x) 
 +CD_\varphi(N) E(\varphi,f)(\cdot,t)^{**}_{N, t^{-1}}(x). 
\end{multline*}   
Using this in \eqref{ineq7} and applying H\"{o}lder's inequality when $q>1$,  
we see that   
\begin{multline} \label{ineq8} 
E(\varphi,f)(\cdot,t)^{**}_{N, t^{-1}}(x)^q\leq C\delta^{-Nq} 
M(|f*\varphi_t|^r)(x)^{q/r} 
\\ 
+  C_q\delta^q\sum_{j\geq 0}
C_\varphi(\nabla\varphi, j,N)^qb^{-jNq}b^{-\tau c_qj} 
 E(\varphi,f)(\cdot,b^jt)^{**}_{N, (b^jt)^{-1}}(x)^q        
\\ 
+C\delta^q D_\varphi(N)^q E(\varphi,f)(\cdot,t)^{**}_{N, t^{-1}}(x)^q. 
\end{multline} 
where $\tau>0$,  $c_q=1$ if $q>1$ and $c_q=0$ if $0<q\leq 1$.  
\par 
If we integrate both sides of the inequality \eqref{ineq8} over $(0,\infty)$ 
with respect to the measure $dt/t$ and if we apply termwise integration on the 
right hand side, then we have  
\begin{multline}\label{ineq9}  
\int_0^\infty E(\varphi,f)(\cdot,t)^{**}_{N, t^{-1}}(x)^q \, \frac{dt}{t} 
\leq 
C\delta^{-Nq} \int_0^\infty M(|f*\varphi_t|^r)(x)^{q/r}\, \frac{dt}{t} 
\\ 
+ C_q\delta^q\left[\sum_{j\geq 0}
C_\varphi(\nabla\varphi, j,N)^qb^{-jNq}b^{-\tau c_qj}
+ D_\varphi(N)^q \right]\int_0^\infty 
E(\varphi,f)(\cdot,t)^{**}_{N, t^{-1}}(x)^q \, \frac{dt}{t}.                   
\end{multline}  
The condition \eqref{ineq5} with $L=N$ implies that the sum in $j$ on the right hand side of \eqref{ineq9} is finite if $\tau$ is small enough.    
 We can see that 
the last integral on the right hand side of \eqref{ineq9} is  finite 
for $f\in \mathscr S(\Bbb R^n)$ by \eqref{beta} and \eqref{beta+}. 
Further, we have \eqref{ineq5+} for $L=N$.  Altogether, 
 it follows that the second term on the right hand side of \eqref{ineq9} 
 is finite. 
Thus, we can get the conclusion if we choose $\delta$ sufficiently small. 
\end{proof} 

\begin{proof}[Proof of Proposition $\ref{T2.3}$] By \eqref{ineq2}  we have 
\begin{multline*} 
|E(\psi,f)(x,t)|^q 
\leq C_q\sum_{j: b^j\leq A} 
C(\psi,j,N)^q b^{-\tau c_qj} 
E(\varphi,f)(\cdot,b^jt)^{**}_{N, (b^jt)^{-1}}(x)^q  
\\   
+CD(\Theta,A,N)^q E(\varphi,f)(\cdot,t)^{**}_{N, t^{-1}}(x)^q,   
\end{multline*}  
where $\tau>0$ and $c_q$ is as in \eqref{ineq8}. 
Integrating with the measure $dt/t$ over $(0,\infty)$, we have 
\begin{multline}\label{ineq6}  
\int_0^\infty |E(\psi,f)(x,t)|^q\, \frac{dt}{t} 
\\ 
\leq C_q\left[\sum_{j: b^j\leq A}
C(\psi,j,N)^q b^{-\tau c_qj} + D(\Theta,A,N)^q\right]  \int_0^\infty 
E(\varphi,f)(\cdot,t)^{**}_{N, t^{-1}}(x)^q \, \frac{dt}{t}.  
\end{multline}  
 The sum in $j$ on the right hand side of 
\eqref{ineq6}  is finite by \eqref{ineq4} with $L=N$ if $\tau$ is small enough; also we have assumed
 $D(\Theta,A,N)<\infty$ (\eqref{ineq4+} with $L=N$).     
Let $r=n/N<q, p$ and $w\in A_{pN/n}$.  
By \eqref{ineq6} and Lemma \ref{L2.4} we see that 
\begin{align}\label{ineq11} 
\left(\int_{\Bbb R^n} \left(\int_0^\infty \right.\right. & \left.\left.
|E(\psi,f)(x,t)|^{q}\, \frac{dt}{t}
\right)^{p/q} w(x)\, dx\right)^{1/p} 
\\
&\leq C\left\|\left(\int_0^\infty M(|f*\varphi_t|^r)(x)^{q/r} 
\, \frac{dt}{t}\right)^{1/q}\right\|_{p,w}    \notag 
\\ 
&= C\left\|\left(\int_0^\infty M(|f*\varphi_t|^r)(x)^{q/r} 
\, \frac{dt}{t}\right)^{r/q}\right\|_{p/r,w}^{1/r}            \notag 
\\
&\leq C\left(\int_{\Bbb R^n} \left(\int_0^\infty |E(\varphi,f)(x,t)|^{q}
\, \frac{dt}{t}\right)^{p/q} w(x)\, dx\right)^{1/p},   \notag 
\end{align}  
where the last inequality follows form the following lemma, which is 
a version of the vector valued inequality for the Hardy-Littlewood maximal 
functions of Fefferman-Stein \cite{FeS} (see \cite{RRT} for a proof of the
 $\ell^\mu$-valued case, which may be available also in the present 
 situation).     

\begin{lemma}\label{L2.6} 
 Suppose that $1<\mu, \nu <\infty$ and $w\in A_\nu$. Then 
for appropriate functions $E(x,t)$ on $\Bbb R^n\times (0, \infty)$ we have 
$$\left\|\left(\int_0^\infty M(E^t)(x)^{\mu} 
\, \frac{dt}{t}\right)^{1/\mu}\right\|_{\nu,w} 
\leq C \left(\int_{\Bbb R^n} \left(\int_0^\infty |E(x,t) |^{\mu}
\, \frac{dt}{t}\right)^{\nu/\mu} w(x)\, dx\right)^{1/\nu},    $$
where $E^t(x)=E(x,t)$.   
\end{lemma}  
This completes the proof of Proposition \ref{T2.3}.  
\end{proof} 
We have an analogous result for general $\varphi=
(\varphi^{(1)}, \dots, \varphi^{(M)})$, although Proposition \ref{T2.3}
 is stated only for the case $M=1$. 
\par 
It is obvious that $Q$, $\hat{Q}(\xi)=-2\pi|\xi|e^{-2\pi|\xi|}$, 
 satisfies all the requirements on $\varphi$ in Lemma \ref{L2.4} for all 
 $N>0$.  
To state results with more directly verifiable assumptions on $\varphi$ and 
$\psi$, we introduce a class of functions. 
\begin{definition} 
Let $\psi\in L^1(\Bbb R^n)$.  Let $l$ be a  non-negative integer and $\tau$ 
a non-negative real number. 
We say $\psi \in B^l_\tau$ if $\hat{\psi}\in C^{l}(\Bbb R^n\setminus\{0\})$ 
and  
$$|\partial_\xi^\gamma \hat{\psi}(\xi)|\leq 
C_\gamma|\xi|^{-\tau-|\gamma|}  
\quad \text{outside a neighborhood of the origin}$$ 
for every $\gamma$ satisfying $|\gamma|\leq l$ with a constant $C_\gamma$, 
where 
$\gamma=(\gamma_1, \dots, \gamma_n)$ is a multi-index, 
$\gamma_j\in \Bbb Z$, $\gamma_j\geq 0$, $|\gamma|=\gamma_1+\dots+\gamma_n$  
and $\partial_\xi^\gamma=
\partial_{\xi_1}^{\gamma_1}\dots \partial_{\xi_n}^{\gamma_n}$.  
\end{definition}  
Clearly, $Q\in B^l_\tau$ for any $l, \tau$.  This is also the case for 
$\psi\in \mathscr S(\Bbb R^n)$.  

\begin{lemma}\label{L2.8} 
Suppose that $\varphi\in L^1(\Bbb R^n)$ and $\varphi$ satisfies 
the condition \eqref{nondegeneracy}. 
Let $\tau\geq 0$, $J>0$ and let $L$ be a non-negative integer. 
\begin{enumerate} 
\item[$(1)$]  
Suppose that 
$\psi\in B^{L+[n/2]+1}_\tau$ and  $\hat{\varphi}
 \in C^{L+[n/2]+1}(\Bbb R^n\setminus\{0\})$, where $[a]$ denotes the largest 
integer not exceeding $a$.  
Then we have 
\begin{equation*}
\sup_{j: b^j\leq J}C_\varphi(\psi,j,L)b^{-j\tau}<\infty,    
\end{equation*} 
where $C_\varphi(\psi,j,L)=C(\psi,j,L)$ is as in \eqref{c}.    
\item[$(2)$]  
Suppose that $\Theta \in C^\infty(\Bbb R^n)$   
and  $\hat{\varphi} \in C^{L+[n/2]+1}(\Bbb R^n\setminus\{0\})$.  
Then 
\begin{equation*}
 D_\varphi(\Theta,J, L)< \infty,       
\end{equation*} 
where  $D_\varphi(\Theta,J, L)=D(\Theta,J, L)$ is as in \eqref{d2}. 
\item[$(3)$] 
Let $\psi^{(k)}\in L^1(\Bbb R^n)$ and 
$\mathscr F(\psi^{(k)})(\xi)=2\pi i\xi_k \hat{\varphi}(\xi)$, $1\leq k\leq n$. 
 If $\varphi\in B^{L+[n/2]+1}_{L+1+\tau}$,  then we have 
\begin{equation*}
\sup_{j: b^j\leq J}C_\varphi(\psi^{(k)},j,L)b^{-jL-j\tau}<\infty, \quad 
D_\varphi(\Xi_k,1, L)<\infty    
\end{equation*} 
for each $k$, where $\Xi_k(\xi)=2\pi i\xi_k$ as above.   
\end{enumerate}
\end{lemma}  
\begin{proof} 
Part (3) follows from part (1) and part (2) since 
$\psi^{(k)} \in B^{L+[n/2]+1}_{L+\tau}$ and 
$\hat{\varphi} \in C^{L+[n/2]+1}(\Bbb R^n\setminus\{0\})$  
if $\varphi \in B^{L+[n/2]+1}_{L+1+\tau}$. 
To prove part (1),   
we note that 
\begin{multline*} 
(1+|x|)^{[n/2]+1}C_0(\psi,t,L,x) 
\\ 
\leq C\left|\int \hat{\psi}(t^{-1}\xi)\hat{\eta}(\xi)e^{2\pi i
\langle x, \xi\rangle}\, d\xi\right|+
C\sup_{|\gamma|=L+[n/2]+1}
\left|\int \partial_\xi^\gamma\left[\hat{\psi}(t^{-1}\xi)
\hat{\eta}(\xi)\right]e^{2\pi i\langle x, \xi\rangle}\, d\xi\right|,     
\end{multline*}  
where $C_0(\psi,t,L,x)$ is as in \eqref{czero}. We note that 
$\hat{\eta} \in C^{L+[n/2]+1}(\Bbb R^n)$ by Lemma \ref{L2.1}, since 
$\hat{\varphi} \in C^{L+[n/2]+1}(\Bbb R^n\setminus\{0\})$.   
The assumption  $\psi\in B^{L+[n/2]+1}_\tau$ implies 
$$\left|\partial_\xi^\gamma \left[\hat{\psi}(t^{-1}\xi)
\hat{\eta}(\xi)\right]\right|\leq C_Mt^{\tau}, \quad 0<t\leq M, $$  
for any $M>0$, 
if $|\gamma|= L+[n/2]+1$ or $\gamma=0$.  It follows that 
\begin{equation*} 
C_0(\psi,t,L,x)\leq C(1+|x|)^{-[n/2]-1}G(x) 
\end{equation*}
with some $G\in L^2$ such that $\|G\|_2\leq Ct^{\tau}$.  Thus, 
since $[n/2]+1>n/2$, by the Schwarz inequality we have   
\begin{equation}\label{ineq12}
\int_{\Bbb R^n}C_0(\psi,t,L,x)\, dx \leq  Ct^{\tau}.    
\end{equation} 
The conclusion of part (1) follows from \eqref{ineq12} with $t=b^j$. 
\par 
Likewise, we have  
\begin{equation*}\label{ineq13}
\int_{\Bbb R^n}D(\Theta,J,L,x)\, dx <\infty    
\end{equation*} 
under the assumptions of part (2), 
where  $D(\Theta,J,L,x)$ is as in \eqref{d}, which proves part (2).   
\end{proof} 
By Lemma \ref{L2.8} and Proposition \ref{T2.3} we have the following. 
\begin{theorem}\label{C2.9} 
Let $\varphi \in L^1(\Bbb R^n)$ satisfy \eqref{nondegeneracy} with $M=1$.  
Suppose that $\psi \in L^1(\Bbb R^n)$ and 
$\hat{\psi}(\xi)=\hat{\varphi}(\xi)\Theta(\xi)$ in a neighborhood 
of the origin with some $\Theta\in C^\infty(\Bbb R^n)$. 
Let $0<p, q<\infty$ and let 
$N$ be a positive integer such that $N>\max(n/p,n/q)$.  Let $w\in A_{pN/n}$. 
Suppose that 
$\varphi$ belongs to $B^{N+[n/2]+1}_{N+1+\epsilon}$ for some $\epsilon>0$ and  
satisfies \eqref{alpha} and \eqref{beta}. 
Also, suppose that 
$\psi\in B^{N+[n/2]+1}_{\epsilon}$ for some $\epsilon>0$. 
Then we have 
$$\left\|\left(\int_0^\infty|f*\psi_t|^q\, \frac{dt}{t}
\right)^{1/q}\right\|_{p,w} 
\leq C\left\|\left(\int_0^\infty|f*\varphi_t|^q\, 
\frac{dt}{t}\right)^{1/q}\right\|_{p,w}   $$  
for $f\in \mathscr S(\Bbb R^n)$, where  $C$ is a positive constant 
independent of $f$.  
\end{theorem} 
\begin{proof} 
If we have \eqref{alpha} and if $\varphi\in 
B^{N+[n/2]+1}_{N+1+\epsilon}$, then \eqref{ineq5} and \eqref{ineq5+} 
hold  with $L=N$ by part (3) of Lemma \ref{L2.8} with $J=1$, 
$\tau=\epsilon$, $L=N$.    Since 
$\psi\in B^{N+[n/2]+1}_{\epsilon}$ and $\varphi \in C^{N+[n/2]+1}(\Bbb R^n
\setminus \{0\})$,  if $\hat{\psi}(\xi)=\hat{\varphi}(\xi)\Theta(\xi)$ 
on $\{|\xi|<r_2A^{-1}\}$, $A\geq 1$, 
 we have \eqref{ineq4} and \eqref{ineq4+} with 
$L=N$ by part (1) of Lemma \ref{L2.8} with $J=A$, $\tau=\epsilon$, $L=N$ and  
part (2) of Lemma \ref{L2.8} with $J=A$, $L=N$, respectively. 
Thus Proposition \ref{T2.3} implies the conclusion.  
\end{proof} 

This immediately implies the following. 

\begin{theorem}\label{C2.10} 
Let $\varphi\in L^1(\Bbb R^n)$ satisfy \eqref{nondegeneracy} with $M=1$, 
\eqref{alpha} and \eqref{beta}.  
We assume that $0<p, q<\infty$ and $N$ is a positive integer satisfying 
$N>\max(n/p,n/q)$.  Let $w\in A_{pN/n}$.   Suppose that 
$\varphi\in B^{N+[n/2]+1}_{N+1+\epsilon}$ for some $\epsilon>0$. 
Then, if $\psi\in \mathscr S(\Bbb R^n)$ and $\hat{\psi}$ vanishes in a 
neighborhood of  the origin,  the inequality 
$$\left\|\left(\int_0^\infty|f*\psi_t|^q\, \frac{dt}{t}
\right)^{1/q}\right\|_{p,w} 
\leq C\left\|\left(\int_0^\infty|f*\varphi_t|^q\, 
\frac{dt}{t}\right)^{1/q}\right\|_{p,w}, \quad f\in \mathscr S(\Bbb R^n),   
$$  
holds with a positive constant $C$ independent of $f$.
\end{theorem} 
\begin{proof} We see that 
$\hat{\psi}(\xi)=\hat{\varphi}(\xi)\Theta(\xi)$ in a neighborhood 
of the origin with $\Theta$ being identically $0$.  
Obviously, $\psi\in B^{N+[n/2]+1}_{1}$. So all the requirements for 
$\varphi$ and $\psi$ in Theorem \ref{C2.9} are satisfied. Thus the 
conclusion follows from Theorem \ref{C2.9}.  This completes the proof.
\end{proof} 
We note that $Q$ fulfills all the requirements on $\varphi$  in 
Theorem \ref{C2.10} for every $N$.  
The same is true of $\varphi_0 \in \mathscr 
S(\Bbb R^n)$ satisfying \eqref{nondegeneracy} (with $M=1$) and  
\eqref{cancell}.

\section{Littlewood-Paley operators and Hardy spaces}  

Let $\mathscr H$ denote the Hilbert space of 
functions $u(t)$ on $(0,\infty)$ such that $\|u\|_{\mathscr H}=
\left(\int_0^\infty|u(t)|^2\, dt/t\right)^{1/2}<\infty$.  
We first recall Hardy spaces of functions on $\Bbb R^n$ with 
values in $\mathscr H$, which will be used to prove \eqref{reverse} by 
Theorem \ref{C2.10} (see Corollary \ref{C3.1} below).  
\par 
The Lebesgue space $L^q_{\mathscr H}(\Bbb R^n)$ consists of functions 
$h(y,t)$ with the norm  
$$\|h\|_{q,\mathscr H}=\left(\int_{\Bbb R^n}\|h^y\|_{\mathscr H}^q \, dy
\right)^{1/q}, $$ 
where $h^y(t)=h(y,t)$. 
For $0<p\leq 1$,  we consider the Hardy space $H^p_{\mathscr H}(\Bbb R^n)$ of 
 functions on $\Bbb R^n$ with values in $\mathscr H$.  
We take $\varphi\in \mathscr S(\Bbb R^n)$ with $\int \varphi(x)\, dx=1$. 
Let $h \in L^2_{\mathscr H}(\Bbb R^n)$.  
We recall that  $h \in H^p_{\mathscr H}(\Bbb R^n)$ if 
$\|h\|_{H^p_{\mathscr H}}= \|h^*\|_{L^p}<\infty$, where  
$$h^*(x)=\sup_{s>0}\left(\int_0^\infty |\varphi_s*h^{t}(x)|^2\, \frac{dt}{t}
\right)^{1/2}, $$ with $h^{t}(x)=h(x,t)$. 
\par 
If $a$ is a $(p, \infty)$ atom in  $H^p_{\mathscr H}(\Bbb R^n)$, we have  
\begin{enumerate} 
\item[(i)] $\left(\int_0^\infty |a(x,t)|^2\, dt/t \right)^{1/2} \leq 
|Q|^{-1/p}$, where $Q$ is a cube in $\Bbb R^n$ with sides parallel to the 
coordinate axes; 
\item[(ii)] $\sup(a(\cdot, t))\subset Q$ uniformly in $t>0$, where $Q$ is the
 same as in $(i);$  
\item[(iii)] $\int_{\Bbb R^n} a(x,t)x^\gamma \, dx=0$ for all $t>0$ and 
$\gamma$ such that $|\gamma|\leq [n(1/p -1)]$, where    
$\gamma=(\gamma_1, \dots, \gamma_n)$ is a multi-index and  
$x^\gamma=x_1^{\gamma_1}\dots x_n^{\gamma_n}$. 
\end{enumerate}  
\par 
We apply the following atomic decomposition.  
\begin{lemma}\label{L3.4}
Let $h\in L^2_{\mathscr H}(\Bbb R^n)$. If 
 $h\in H^p_{\mathscr H}(\Bbb R^n)$, 
then there exist a sequence $\{a_k\}$ of 
$(p,\infty)$ atoms in $H^p_{\mathscr H}(\Bbb R^n)$ and a sequence 
$\{\lambda_k\}$ of positive numbers such that 
$\sum_{k=1}^\infty\lambda_k^p\leq C\|h\|_{H^p_{\mathscr H}}^p$ with 
a constant $C$ independent of $h$ and 
$h=\sum_{k=1}^\infty\lambda_k a_k$ in $H^p_{\mathscr H}(\Bbb R^n)$ and in 
$L^2_{\mathscr H}(\Bbb R^n)$.   
\end{lemma} 
A proof of the atomic decomposition for $H^p(\Bbb R^n)$ can be found in 
\cite{GR} and \cite{ST}.    
 Similar methods apply to the vector valued case.     
\par 
In this section, we prove the following result 
as an application of Theorem \ref{C2.10}. 
\begin{corollary}\label{C3.1} 
Let $0<p\leq 1$, $N>n/p$. 
Suppose that $\varphi \in L^1(\Bbb R^n)$ satisfies \eqref{nondegeneracy} with 
$M=1$, \eqref{alpha},  \eqref{beta} and suppose that 
$\varphi\in B^{N+[n/2]+1}_{N+1+\epsilon}$ for some $\epsilon>0$. 
Then we have 
$$ \|f\|_{H^p}\leq C_p\|g_\varphi(f)\|_p  $$  
for $f\in H^p(\Bbb R^n)\cap \mathscr S(\Bbb R^n)$, where $C_p$ is a positive 
constant independent of $f$. 
\end{corollary} 
This can be generalized to an arbitrary $f\in H^p(\Bbb R^n)$ if $\varphi=Q$ or 
if $\varphi$ is a function in $\mathscr S(\Bbb R^n)$ satisfying 
\eqref{nondegeneracy} and \eqref{cancell} (see \cite{U}). 
\par 
 In proving Corollary \ref{C3.1}, we need the following.  
\begin{lemma}\label{L3.5} 
Suppose that $\eta\in \mathscr S(\Bbb R^n)$, $\supp(\hat{\eta})\subset 
\{1/2\leq |\xi|\leq 4\}$,  $\hat{\eta}(\xi)=1$ on 
$\{1\leq |\xi|\leq 2\}$ and   
that $\Phi\in \mathscr S(\Bbb R^n)$ satisfies 
$\int_{\Bbb R^n} \Phi(x)\, dx=1$. 
Let $\psi \in \mathscr S(\Bbb R^n)$ and 
 $\supp \hat{\psi}\subset \{1\leq |\xi|\leq 2\}$.  
Then, for $p, q>0$ and $f\in  \mathscr S(\Bbb R^n)$ we have 
$$\left\|\left(\int_0^\infty\sup_{s>0}|\Phi_s*\psi_t*f|^q\, \frac{dt}{t}
\right)^{1/q}\right\|_{p} 
\leq C\left\|\left(\int_0^\infty|\eta_t*f|^q\, 
\frac{dt}{t}\right)^{1/q}\right\|_{p}. $$  
\end{lemma} 
\begin{proof} 
We note that $\hat{\Phi}(s\xi)\hat{\psi}(t\xi)=\hat{\Phi}(s\xi)\hat{\psi}(t\xi)
\hat{\eta}(t\xi)$. Thus we have 
\begin{align*} 
|\Phi_s*\psi_t*f(x)|&\leq (f*\eta_t)^{**}_{N, t^{-1}}(x)
\int_{\Bbb R^n}|\Phi_s*\psi_t(w)|(1+t^{-1}|w|)^N\, dw 
\\ 
&= (f*\eta_t)^{**}_{N, t^{-1}}(x) 
\int_{\Bbb R^n}|\Phi_{s/t}*\psi(w)|(1+|w|)^N\, dw 
\\ 
&\leq C_N (f*\eta_t)^{**}_{N, t^{-1}}(x)  
\end{align*} 
for any $N>0$, 
with a positive constant $C_N$ independent of $s, t$. The last inequality 
follows from  the observation that 
$\Phi_{s/t}*\psi$, $s, t>0$, belongs to a bounded subset of the 
topological vector space $\mathscr S(\Bbb R^n)$, 
since   $\mathscr F(\Phi_{u}*\psi)(\xi) 
=\hat{\Phi}(u\xi)\hat{\psi}(\xi)$, $u>0$, and $\hat{\psi}(\xi)$  
is supported on $\{1\leq |\xi|\leq 2\}$.   
Therefore, we have 
\begin{equation}\label{ineq3.1}  
\left(\int_0^\infty\sup_{s>0}|\Phi_s*\psi_t*f(x)|^q\, \frac{dt}{t}
\right)^{1/q}
\leq C\left(\int_0^\infty|(f*\eta_t)^{**}_{N, t^{-1}}(x)|^q\, 
\frac{dt}{t}\right)^{1/q}.  
\end{equation}
Thus \eqref{ineq3.1} and Lemma \ref{L2.4} with $\eta$ in place of $\varphi$ 
imply 
\begin{equation*} 
\left(\int_0^\infty\sup_{s>0}|\Phi_s*\psi_t*f(x)|^q\, \frac{dt}{t}
\right)^{1/q}
\leq C\left(\int_0^\infty M(|f*\eta_t|^r)(x)(x)^{q/r}
\, \frac{dt}{t}\right)^{1/q}, 
\end{equation*} 
with $N=n/r$.  By this and Lemma \ref{L2.6}, the conclusion follows 
as in \eqref{ineq11}.   
\end{proof} 
\par 
 We also use the following to prove Corollary \ref{C3.1}. 
\begin{lemma}\label{L3.2} 
Let $\hat{\psi}\in \mathscr S(\Bbb R^n)$ be a radial function supported on 
$\{1\leq |\xi|\leq 2\}$ such that 
$$\int_0^\infty|\hat{\psi}(t\xi)|^2\, \frac{dt}{t}=1 \quad \text{for all 
$\xi\neq 0$.} 
$$  
Let  $f\in H^p(\Bbb R^n)\cap \mathscr S(\Bbb R^n)$, $0<p\leq 1$, and 
put $E(y,t)=f*\psi_t(y)$.  
Then $E$ is in $H^p_{\mathscr H}(\Bbb R^n)$  and we have 
\begin{equation*}\label{}  
\|f\|_{H^p}\leq   C\|E\|_{H^p_{\mathscr H}}.     
\end{equation*} 
\end{lemma} 
Let $\psi$ be a function in $L^1(\Bbb R^n)$ satisfying \eqref{cancell}. 
Suppose that $h\in L^2_{\mathscr H}$. Let 
$h_{(\epsilon)}(y,t)=h(y,t)\chi_{(\epsilon, \epsilon^{-1})}(t)$, 
$0<\epsilon<1$,  where 
$\chi_S$ denotes the characteristic function of a set $S$. 
Put 
\begin{equation*}
F_\psi^\epsilon(h)(x)=\int_0^\infty\int_{\Bbb R^n} 
\psi_t(x-y)h_{(\epsilon)}(y,t)\,dy\,\frac{dt}{t}.  
\end{equation*} 
\par 
To prove Lemma \ref{L3.2} we apply the following.  
\begin{lemma}\label{L3.3} Let $0 <p\leq 1$. 
 Suppose that $\psi\in \mathscr S(\Bbb R^n)$ and $\supp \hat{\psi}\subset 
 \{1\leq |\xi|\leq 2\}$. Then 
$$\sup_{\epsilon\in (0,1)}
\|F_\psi^\epsilon(h)\|_{H^p}\leq C\|h\|_{H^p_{\mathscr H}}.   $$  
\end{lemma} 
\begin{proof} 
Let $a$ be a $(p, \infty)$ atom in  $H^p_{\mathscr H}(\Bbb R^n)$ with support 
in the cube $Q$ of the definition of the atom. 
We denote by $y_0$ the center of $Q$.  Let $\widetilde{Q}$ be a concentric enlargement of $Q$ such that $2|y-y_0|<|x-y_0|$  
if $y\in Q$ and $x\in \Bbb R^n\setminus \widetilde{Q}$.  
Let $\Phi$ be a non-negative $C^\infty$ function on $\Bbb R^n$ 
supported on $\{|x|<1\}$ which satisfies $\int \Phi(x)\, dx=1$. 
Let $\Psi_{s,t}=\Phi_s*\psi_t$, $s, t>0$.  Then  
$\Psi_{s,t}=(\Phi_{s/t}*\psi)_t$ and $\Phi_{u}*\psi$, $u>0$,  belongs to 
a bounded subset of the topological vector space $\mathscr S(\Bbb R^n)$, 
as in the proof of Lemma \ref{L3.5}.  
\par 
Let $P_x(y,y_0)$ be the Taylor polynomial in $y$ of order $M=[n(1/p -1)]$ 
at $y_0$ for $\Phi_{s/t}*\psi(x-y)$. Then,  if 
$|x-y_0|>2|y-y_0|$, we see that 
$$|\Phi_{s/t}*\psi(x-y)-P_x(y,y_0)|\leq C|y-y_0|^{M+1}
(1+|x-y_0|)^{-L},   $$  
where $L>n+M+1$ and the constant $C$ is independent of $s, t, x, y, y_0$, 
and hence  
$$|\Psi_{s,t}(x-y)-t^{-n}P_{x/t}(y/t,y_0/t)|\leq Ct^{-n-M-1}|y-y_0|^{M+1}
(1+|x-y_0|/t)^{-L}.   $$  
Therefore, by the properties of an atom and the Schwarz 
inequality, for $x\in \Bbb R^n\setminus \widetilde{Q}$ we have  
\begin{align*} 
&\left|\Phi_s*F_\psi^\epsilon(a)(x)\right|=\left|\iint_{\Bbb R^n \times 
(0,\infty)}   
\left(\Psi_{s,t}(x-y)-t^{-n}P_{x/t}(y/t,y_0/t) \right)
a_{(\epsilon)}(y,t)\, dy\, \frac{dt}{t}\right|  
\\ 
&\leq \int_Q\left(\int_0^\infty\left|\Psi_{s,t}(x-y)-t^{-n}P_{x/t}(y/t,y_0/t) 
\right|^2 \, \frac{dt}{t}\right)^{1/2}\left(\int_0^\infty|a(y,t)|^2
\, \frac{dt}{t}\right)^{1/2}\, dy
\\ 
&\leq C|Q|^{-1/p}\int_{Q}\left(\int_0^\infty
\left|\Psi_{s,t}(x-y)-t^{-n}P_{x/t}(y/t,y_0/t) \right|^2 \, 
\frac{dt}{t}\right)^{1/2}\, dy 
\\ 
&\leq C|Q|^{-1/p}\int_{Q}|y-y_0|^{M+1} |x-y_0|^{-n-M-1}
\, dy 
\\ 
&\leq C|Q|^{-1/p+ 1+(M+1)/n}|x-y_0|^{-n-M-1}. 
\end{align*}  
We note that $p>n/(n+M+1)$. Thus    
\begin{multline}\label{a1}
\int_{\Bbb R^n\setminus \widetilde{Q}} \sup_{s>0}
\left|\Phi_s*F_\psi^\epsilon(a)(x)
\right|^p\, dx 
\\ 
\leq C|Q|^{-1+ p+p(M+1)/n}\int_{\Bbb R^n\setminus \widetilde{Q}}
|x-y_0|^{-p(n+M+1)} \leq C. 
\end{multline} 
\par 
Since $\int_0^\infty|\hat{\psi}(t\xi)|^2\, dt/t\leq C$,  by duality 
we have  
$$\sup_{\epsilon\in(0,1)}\|F_\psi^\epsilon(h)\|_2 
\leq C\|h\|_{L^2_\mathscr H}, \quad h\in L^2_{\mathscr H}(\Bbb R^n).$$ 
Thus, applying H\"{o}lder's inequality, by the properties (i), (ii) of $a$ 
 we see that 
\begin{align}\label{a2}
\int_{\widetilde{Q}} \sup_{s>0}\left|\Phi_s*F_\psi^\epsilon(a)(x)
\right|^p\, dx 
&\leq C|Q|^{1-p/2}\left(\int_{\widetilde{Q}}|M(F_\psi^\epsilon(a))(x)|^2
\, dx \right)^{p/2} 
\\ 
&\leq C|Q|^{1-p/2}\left(\int_{Q}\int_0^\infty |a(y,t)|^2\, \frac{dt}{t} \, dy 
\right)^{p/2}    \notag 
\\ 
&\leq C.    \notag
\end{align} 
The estimates \eqref{a1} and \eqref{a2} imply 
\begin{equation}\label{a3}
\int_{\Bbb R^n} \sup_{s>0}\left|\Phi_s*F_\psi^\epsilon(a)(x)\right|^p
\, dx \leq C. 
\end{equation} 
Using Lemma \ref{L3.4} and \eqref{a3},  we have  
 \begin{equation*}\label{h1}
\int_{\Bbb R^n} \sup_{s>0}\left|\Phi_s*F_\psi^\epsilon(h)(x)\right|^p
\, dx \leq C\|h\|_{H^p_{\mathscr H}}^p.   
\end{equation*} 
This completes the proof. 
\end{proof} 

\begin{proof}[Proof of Lemma $\ref{L3.2}$.]  
The fact that $E\in H^p_{\mathscr H}(\Bbb R^n)$ can be proved similarly 
to the proof of Lemma \ref{L3.3} by applying the atomic decomposition for 
$H^p(\Bbb R^n)$ (see \cite[Lemma 3.6]{U}). 
\par  
We write  
\begin{equation*} 
F^\epsilon_{\widetilde{\bar{\psi}}}(E)(x)=\int_\epsilon^{\epsilon^{-1}}
\int_{\Bbb R^n} \psi_t*f(y)\bar{\psi}_t(y-x)\,dy\,\frac{dt}{t} 
= \int_{\Bbb R^n} \Psi^{(\epsilon)}(x-z)f(z)\, dz, 
\end{equation*} 
where $\bar{\psi}$ denotes the complex conjugate, $\widetilde{g}(x)=g(-x)$ and 
\begin{equation*} 
\Psi^{(\epsilon)}(x)=
\int_\epsilon^{\epsilon^{-1}}\int_{\Bbb R^n}  
\psi_t(x+y)\bar{\psi}_t(y)\,dy\,\frac{dt}{t}. 
\end{equation*} 
We note that 
\begin{equation*} 
\widehat{\Psi^{(\epsilon)}}(\xi)
=\int_\epsilon^{\epsilon^{-1}}\hat{\psi}(t\xi)\widehat{\bar{\psi}}(-t\xi)
\,\frac{dt}{t}=\int_\epsilon^{\epsilon^{-1}}|\hat{\psi}(t\xi)|^2
\,\frac{dt}{t}. 
\end{equation*} 
From this and Lemma \ref{L3.3} we have 
\begin{equation*}\label{}  
\|f\|_{H^p}\leq  
 C\liminf_{\epsilon\to 0}\|F^\epsilon_{\widetilde{\bar{\psi}}}(E)\|_{H^p}
\leq C\|E\|_{H^p_{\mathscr H}}.   
\end{equation*} 
\end{proof} 
\begin{proof}[Proof of Corollary $\ref{C3.1}$.] 
We take a function  $\eta$  as in Lemma \ref{L3.5}.  
Then by Lemma \ref{L3.5} with $q=2$ and Lemma \ref{L3.2},
it follows that  
$$\|f\|_{H^p}\leq  C\left\|g_\eta(f)\right\|_p 
  $$  
  for $f\in H^p(\Bbb R^n)\cap \mathscr S(\Bbb R^n)$. 
If we use this and Theorem \ref{C2.10} with $q=2$, $w=1$ and with $\eta$ 
in place of $\psi$, we can reach the conclusion of Corollary \ref{C3.1}. 
\end{proof} 
\par 
We can also prove discrete parameter versions of 
Proposition \ref{T2.3} and Corollary \ref{C3.1} by analogous methods. 
\begin{proposition}\label{P3.6} 
Let $N>0$, $n/N<p, q<\infty$.   
Suppose that $w\in A_{pN/n}$ and that $\varphi$ and $\psi$ fulfill 
the hypotheses of Proposition 
$\ref{T2.3}$ with $N$.  Then, for $f\in \mathscr S(\Bbb R^n)$ we have 
$$\left\|\left(\sum_{j=-\infty}^\infty|f*\psi_{b^j}|^q\right)^{1/q}
\right\|_{p,w} 
\leq 
C\left\|\left(\sum_{j=-\infty}^\infty|f*\varphi_{b^j}|^q
\right)^{1/q}\right\|_{p,w}.   
$$
\end{proposition}  

\begin{corollary}
Let $0<p\leq 1$ and $N>n/p$. Suppose that $\varphi$ fulfills the 
hypotheses of Corollary $\ref{C3.1}$ with $N$.  Then, for 
$f\in H^p(\Bbb R^n)\cap \mathscr S(\Bbb R^n)$ we have  
$$\|f\|_{H^p}\leq C\left\|\left(\sum_{j=-\infty}^\infty|f*\varphi_{b^j}|^2
\right)^{1/2}\right\|_p.  $$  
\end{corollary}  
Also, from Proposition \ref{P3.6} we have discrete parameter analogues of  
Theorems \ref{C2.9} and \ref{C2.10}. 

\section{Proofs of Lemmas $\ref{L2.1}$ and $\ref{L2.5}$ }  

In this section we give proofs of Lemmas $\ref{L2.1}$ and $\ref{L2.5}$ 
for completeness. 
\begin{proof}[Proof of Lemma $\ref{L2.1}$.]  
There exist a finite family $\{I_j\}_{j=1}^L$ of compact intervals in 
$(0,\infty)$ and a positive constant $c$ such that 
\begin{equation*} 
\inf_{\xi\in S^{n-1}}\max_{1\leq j\leq L}\inf_{t\in I_j}\sum_{i=1}^M
|\mathscr F(\varphi^{(i)})(t\xi)|^2\geq c,    
\end{equation*}  
where $S^{n-1}=\{\xi: |\xi|=1\}$.  
This follows from a compactness argument, since each 
$\mathscr F(\varphi^{(j)})$ is continuous.   
\par 
 Let $b_0=\max_{1\leq h\leq L}(a_h/b_h)$, where $I_h=[a_h,b_h]$.  
 Then $b_0\in (0,1)$ and if $b\in [b_0,1)$, 
$t>0$ and $1\leq h\leq L$, $h\in \Bbb Z$, we have $b^jt\in I_h$ for 
some $j\in \Bbb Z$.
\par 
Let $[m,H]$ be an interval in $(0, \infty)$ such that 
$\cup_{j=1}^L I_j\subset [m,H]$.   We take a non-negative function  
$\theta\in C_0^\infty(\Bbb R)$ such that $\theta=1$ on $[m,H]$, 
$\supp \theta\subset [m/2,2H]$.  Then 
$$\sum_{j=-\infty}^\infty\theta(b^j|\xi|)\sum_{i=1}^M
|\mathscr F(\varphi^{(i)})(b^j\xi)|^2=:\Psi(\xi) \geq c>0 \quad 
\text{for $\xi\neq 0$}. $$  
We have $\Psi(b^k\xi)=\Psi(\xi)$ for $k\in \Bbb Z$.  
Define 
$$\mathscr F(\eta^{(j)})(\xi)=\theta(|\xi|)
\overline{\mathscr F(\varphi^{(j)})(\xi)} 
\Psi(\xi)^{-1}\quad \text{for $\xi\neq 0$} $$ 
and $\mathscr F(\eta^{(j)})(0)=0$. Then, $\eta$ has all the properties 
required in the lemma. Also, from the construction, we can see that 
$\hat{\eta}\in C^k(\Bbb R^n)$ 
if $\hat{\varphi}\in C^k(\Bbb R^n\setminus\{0\})$.  
This completes the proof. 
\end{proof} 

\begin{proof}[Proof of Lemma $\ref{L2.5}$.]  
Let 
$\dashint_{B(x,t)} f(y)\,dy=|B(x,t)|^{-1}\int_{B(x,t)} f(y)\,dy$, where 
 $B(x,t)$ denotes a ball in $\Bbb R^n$ with center $x$ and radius $t$. 
Then, for $u, r>0$ and $x, z\in \Bbb R^n$, 
\begin{align*} 
|F(x-z)|
&=\left(\dashint_{B(x-z,u)}|F(y)+(F(x-z)-F(y))|^r\, dy\right)^{1/r} 
\\ 
&\leq C_r\left(\dashint_{B(x-z,u)}|F(y)|^r\, dy\right)^{1/r} 
+C_r\left(\dashint_{B(x-z,u)}|F(x-z)-F(y)|^r\, dy\right)^{1/r},  
\end{align*} 
where $C_r=1$ if $r\geq 1$ and $C_r=2^{-1+1/r}$ if $0<r< 1$.  Therefore 
\begin{equation}\label{4.1}  
 |F(x-z)| \leq C_r\left(\dashint_{B(x-z,u)}|F(y)|^r\, dy\right)^{1/r} 
+C_r  \sup_{y:|x-z-y|<u}u |\nabla F(y)|.  
\end{equation} 
\par 
If $|x-z-y|<u$, $|x-y|<u+|z|$.  Thus we have  
\begin{align*} 
|\nabla F(y)|&\leq \frac{|\nabla F(x+(y-x))|}{(1+R|x-y|)^N}(1+R(u+|z|))^N 
\\ 
&\leq |\nabla F|_{N,R}^{**}(x)(1+\delta +R|z|)^N 
\\ 
&\leq 2^N|\nabla F|_{N,R}^{**}(x)(1 +R|z|)^N 
\end{align*} 
if we choose $u=\delta/R$.  Consequently, 
\begin{equation}\label{4.2} 
\sup_{y: |x-z-y|<u} u|\nabla F(y)|\leq 2^N u|\nabla F|_{N,R}^{**}(x)(1 +R|z|)^N
\end{equation} 
with $u=\delta/R$. 
\par 
Also, if $u=\delta/R$, 
\begin{align}\label{4.3} 
&\left(\dashint_{B(x-z,u)}|F(y)|^r\, dy\right)^{1/r}
\leq \left(u^{-n}(u+|z|)^n\dashint_{B(x,u+|z|)}|F(y)|^r\, dy\right)^{1/r} 
\\ 
&\leq u^{-n/r}(u+|z|)^{n/r} M(|F|^r)(x)^{1/r}     \notag 
\\ 
&=\delta^{-n/r}(\delta+R|z|)^{n/r} M(|F|^r)(x)^{1/r} \notag 
\\ 
&\leq \delta^{-n/r}(1+R|z|)^{n/r} M(|F|^r)(x)^{1/r}.  \notag 
\end{align} 
\par 
From \eqref{4.1}, \eqref{4.2} and \eqref{4.3}, we see that  
\begin{equation*} 
|F(x-z)|\leq C_r \delta^{-n/r}(1+R|z|)^{n/r} M(|F|^r)(x)^{1/r} + 
2^N C_r u|\nabla F|_{N,R}^{**}(x)(1 +R|z|)^N.  
\end{equation*} 
 If $N=n/r$, it follows that 
\begin{equation*} 
\frac{|F(x-z)|}{(1 +R|z|)^N}\leq C_r \delta^{-N} M(|F|^r)(x)^{1/r} + 
2^N C_r \delta R^{-1}|\nabla F|_{N,R}^{**}(x).  
\end{equation*} 
Thus we have the conclusion of the lemma by 
taking the supremum in $z$ over $\Bbb R^n$. 
\end{proof}

\end{document}